\documentclass[11pt, letterpaper]{article}
\usepackage[margin = 1in]{geometry}
\usepackage[utf8]{inputenc}
\usepackage{amsmath}
\usepackage{amssymb}
\usepackage{booktabs}
\usepackage{mathrsfs}
\usepackage{url}
\usepackage{bm}
\usepackage[
  separate-uncertainty = true,
  multi-part-units = repeat
]{siunitx}
\usepackage{mathtools}
\def\multiset#1#2{\ensuremath{\left(\kern-.3em\left(\genfrac{}{}{0pt}{}{#1}{#2}\right)\kern-.3em\right)}}
\usepackage{tikz}
\usepackage{youngtab}

\usepackage{amsthm}
\usepackage{enumerate}
\usepackage{microtype}
\usepackage{tikz}
\usepackage{thmtools,thm-restate}

\newcommand{\suchthat}{\;\ifnum\currentgrouptype=16 \middle\fi|\;}

\newcommand{\simpc}{\Delta_{n,\ell}}

\newcommand{\inlatA}{\mathcal{L}(\mathcal{A})}
\renewcommand{\a}{\alpha}

\renewcommand{\l}{\lambda}
\newtheorem{theorem}{Theorem}[section]
\newtheorem{lemma}[theorem]{Lemma}
\newtheorem{proposition}[theorem]{Proposition}
\newtheorem{corollary}[theorem]{Corollary}
\newtheorem{conjecture}{Conjecture}
\theoremstyle{definition}

\theoremstyle{definition}
\newtheorem*{example}{Example}

\newcommand{\caseA}{{\noindent \bf Case $\bm{s=1}$.~}}
\newcommand{\caseB}{{\noindent \bf Case $\bm{s=2}$.~}}
\newcommand{\caseC}{{\noindent \bf Case $\bm{s=3}$.~}}

\title{Classifications of $\ell$-Zero-Sumfree Sets}
\author{Ashleigh Adams, Carole Hall, Eric Stucky}
\begin{document}
\maketitle

\begin{abstract}
    The set of all $\ell$-zero-sumfree subsets of $\mathbb{Z}/n\mathbb{Z}$ is a simplicial complex denoted by $\simpc$. We create an algorithm via defining a set of integer partitions we call $(n,\ell)$-congruent partitions in order to compute this complex for moderately-sized parameters $n$ and $\ell$. We also theoretically determine $\simpc$ for several infinite families of parameters, and compute the intersection posets and the characteristic polynomials of the corresponding coordinate subspace arrangements.
\end{abstract}

\section{Introduction}
In 1988, Cameron and Erd\H{o}{s} \cite{CameronErdos} conjectured that the maximum size of a sum-free subset $S\subseteq\left\{1,2,...,n\right\}$ is $O(2^{\frac{n}{2}})$, where a set is ``sum-free'' if no two distinct elements in $S$ add to another element in $S$. This conjecture has since been upgraded to a theorem by \cite{Green} and \cite{Sapozhenko}, and has inspired a great deal of work on similar objects and generalizations. For instance, Calkin and Thomson \cite{CalkinThomson} define the notion of a $(k,\ell)$-sumfree set, which is a set $S\subseteq\mathbb{N}$ such that the equation
$$x_1+x_2+\cdots+ x_k = y_1+y_2+\cdots+y_\ell$$
has no solutions with all $x_i,y_i\in S$.  Notice that when $k=2$ and $\ell=1$, this agrees with the classical definition of a sum-free set. These sets and their close variants have also been studied by many authors; see e.g. \cite{BierChin} for an early study of $(k,\ell)$-sumfree sets in cyclic groups,
and \cite{menu-research-problems} for a detailed survey of recent directions.

In this paper we consider the case of $(k,\ell)$-sumfree sets in $\mathbb{Z}/n\mathbb{Z}$ where $k=0$, which we call \textbf{$\bm{\ell}$-zero-sumfree sets}. We denote the collection of all such sets by $\simpc$; that is,
$$\simpc \coloneqq \Big\{\{s_1,\dots, s_j\}\subset{\mathbb{Z}/n\mathbb{Z}}:c_1s_1 + \cdots + c_js_j\neq 0\text{ whenever } \sum_{i=1}^{j}c_i = \ell\Big\}.$$
This case appears to have qualitatively different behavior than when both $k$ and $\ell$ are strictly positive. For instance, Bajnok and Matzke \cite{BajnokMatzke}, following on the work of \cite{HamidounePlagne}, recently determined an explicit formula for the maximum size of a $(k,\ell)$-sumfree set in $\mathbb{Z}/n\mathbb{Z}$ for any $k,\ell>0$; however, their methods did not directly extend to the $k=0$ case (but see upcoming work of Bajnok, Matzke, and the first author).

In contrast to much of the literature, which focuses on extremal properties of $\Delta_{n,\ell}$, we are concerned with understanding its global structure. To this end, it is extremely useful to have a complete classification of all $\ell$-zero-sumfree sets. In \textbf{Theorem \ref{algorithm}}, we describe an algorithm that we used to compute $\simpc$ for small $n$ and $\ell$, and to observe its properties. The main result of this paper is a complete description of $\simpc$ for three types of parameters: 

\begin{enumerate}
    \item a ``doubling'' class with parameters $n=2^{m+1}\rho$ and $\ell = 2^m\rho$, for odd $\rho\geq{3}$,
    \item a ``prime powers'' class with parameters $n = p^e$ and $\ell = p^e-1$ for prime $p\geq{2}$ and exponent $e\geq{1}$, and 
    \item three ``arms and legs'' classes with parameters $n = 2p$, $\ell = 2p-s$ for prime $p\geq{3}$ and $s\in\left\{1,2,3\right\}$.
\end{enumerate}

For each of these classes, we also use these descriptions to compute some combinatorial and geometric invariants. For instance, to any $\Delta_{n,\ell}$ there is a naturally associated collection $\mathcal{A}$, which consists of finitely many subspaces of a vector space. For the classes described above, we compute this collection, its intersection poset, and its characteristic polynomial.

\section{Simplicial Complexes}

Given a set $V$, a \textbf{simplicial complex} (or just a \textbf{complex}) on $V$ is a nonempty collection $\Gamma$ of subsets which is closed under inclusion; that is, if $S\in\Gamma$ and $T\subseteq S$, then $T\in\Gamma$ as well. Elements of simplicial complexes are called \textbf{faces}, and faces which are maximal under inclusion are called \textbf{facets}. For any face $S\in\Gamma$, its \textbf{dimension} is defined to be $\dim(S)=|S|-1$, and then the \textbf{dimension} of the complex itself is defined to be $\dim(\Gamma) = \displaystyle\max_{S\in\Gamma} \dim(S)$.

Notice that simplicial complexes are entirely determined by their (nonempty) set of facets. Therefore, given any collection $\mathcal F$ of subsets of $V$, we write $\langle\mathcal F\rangle$ to mean the complex $\left\{S\in 2^V: S\subseteq F \text{ for some } F\in \mathcal F\right\}$. If no two elements of $\mathcal F$ contain each other, then $\mathcal F$ is precisely the collection of facets in $\langle\mathcal F\rangle$.

We observe that for any $k,\ell\geq 0$, the collection of $(k,\ell)$-sumfree sets in any abelian group $G$ is a simplicial complex on $G$. This is because if there are no solutions to $x_1+\cdots+x_k=y_1+\cdots y_\ell$ with all $x_i,y_j$ in $S$, then the same is clearly true for any subset of $S$. The structures of these complexes are not well-understood; indeed, most prior research has been devoted merely to computing their dimensions.

In particular, this means that $\Delta_{n,\ell}$ is a simplicial complex for any $0<\ell<{n}$. In order to understand this class of complexes, we developed an algorithm for computing small examples. Key to this algorithm is the notion of the \textbf{Alexander dual}: $\Gamma^\vee = 2^V\setminus \left\{S\in 2^V: S\notin \Gamma \right\}$. Notice that $\Gamma^\vee$ is a complex if and only if $\Gamma$ is a complex, and $(\Gamma^\vee)^\vee=\Gamma$. We also recall that a \textbf{partition} $\lambda$ of $N$ into $\ell$ \textbf{parts} is a non-increasing ordered list of non-negative integers $(\lambda_1,\dots, \lambda_\ell)$ which sum to $N$ (note that an ``ordered list'' allows for repetition of elements).

\begin{theorem}
\label{algorithm}
The following procedure computes $\simpc$, the simplicial complex of of all $\ell$-zero-sumfree subsets of $\mathbb{Z}/n\mathbb{Z}$:
\begin{enumerate}
    \item For each $0\leq m\leq n-1$, generate the partitions of $mn$ into $\ell$ parts which also satisfy $\lambda_1\leq n-1$ and consider their underlying sets. That is, create the set (without multiplicity) $\sigma=\left\{\lambda_1,\dots,\lambda_h\right\}$ with $h\leq{\ell}$ and $\lambda_1\geq\lambda_2\geq\cdots\geq\lambda_h$ from each such partition $\lambda$. Store these $\sigma$ into a set called NLC (short for ``$(n,
    \ell)$-congruent'' partitions).
    \item Take the complement of every element in $NLC$; denote the set of these complements by $\text{NLC}_c$.
    \item Take the Alexander dual of $\langle\text{NLC}_c\rangle$ to obtain $\simpc$: $\langle\text{NLC}_c\rangle^{\lor} = \simpc$.
\end{enumerate}
\end{theorem}
\begin{proof}
Define the set $[\text{NLC}]$ to be the collection of all subsets of our vertex set $V$ such that the subset contains an NLC partition. This is, therefore, the collection of all subsets of $V$ that are \textit{not} $\ell$-zero-sumfree: $[\text{NLC}] = 2^V\setminus{\simpc} = \simpc^c.$

Therefore, we obtain the desired equality in Step 3 as follows:
$$\langle\text{NLC}_c\rangle^{\lor} = \left\{\sigma: \sigma\subseteq\tau \text{ for some }  (2^V\setminus\tau)\in{\text{NLC}}\right\}^{\lor} = \left\{\sigma: (2^V\setminus\sigma)\in{[\text{NLC}]}\right\}^{\lor}.$$
This is, by definition $(\simpc^{\lor})^{\lor}$, which as noted above is $\simpc$.
\end{proof}

\subsection{Disjoint Unions of Simplices}

A \textbf{d-simplex} is the simplicial complex $\Gamma=2^V$ on some set $V$ having $d+1$ elements. Note that the indexing is chosen this way so that a $d$-simplex has dimension $d$. Some of the $\simpc$ that we consider are built from simplices in a particularly simple way: for any two simplicial complexes $\Gamma_1$ and $\Gamma_2$ on disjoint vertex sets $V_1$ and $V_2$, the \textbf{disjoint union}  $\Gamma_1\cup\Gamma_2$ is a simplicial complex on $V_1\cup V_2$. Because we have not found this information elsewhere in the literature, we would like to state some elementary facts about disjoint unions of simplices. 

For any simplicial complex $\Gamma$ with dimension $d$, its \textbf{f-vector} $f(\Gamma)$ is the list of numbers $(f_{-1}, f_0,\dots, f_d)$, where each $f_k$ is the number of faces in $\Gamma$ having dimension $k$. This information is also encoded in its \textbf{h-vector} $h(\Gamma)=(h_0,h_1,\dots, h_{d+1})$, where
$$h_k = \sum_{i=0}^k (-1)^{k-i}\binom{d+1-i}{d+1-k} f_{i-1}.$$

Notice that any simplicial complex contains $\varnothing$, the unique face of dimension $-1$, and so $f_{-1}=h_0=1$.

\begin{proposition} \label{f-vec-distinct-simplices}
If $\Gamma=\Delta_{1}\cup\ldots\cup\Delta_{\a}$ where each $\Delta_{i}$ is a $d_i$-simplex, then the $f$-vector of $\Gamma$ is given by $f_{-1}=1$ and $f_{k-1}=\sum_{i=1}^{\a}\binom{d_i+1}{k}$ for all $1\leq k\leq d$.
\end{proposition}

\begin{proof}
Let $\Gamma$ be a simplicial complex of dimension $d$ such that it is comprised of only $\a$ many $d_i$-simplices, not necessarily distinct. By definition of the binomial coefficient, a $d$-simplex has the $f$-vector given by $f_{k-1}=\binom{d+1}{k}$. One easily checks that the $f$-vector is nearly additive with respect to disjoint unions: $f_k(\Gamma_1\cup\Gamma_2) =f_k(\Gamma_1)+f_k(\Gamma_2)$ for all $k\geq 0$, from which the proposition follows.
\end{proof}


We now turn our attention to the $h$-vector. The non-additivity of $f_{-1}$ for disjoint unions causes more serious difficulties for the $h$-vector, since the defining sum for $h_k$ contains a term with $f_{-1}$ for any $k$. However, there is a surprisingly pleasing formula for the $h$-vector as well.

Given a partition $\l=(\l_1,\dots, \l_\a)$, we may draw its \textbf{Young diagram} (in French notation), a left-and-bottom-justified array of boxes with $\l_i$ boxes in row $i$.
For example, if $\l=(5,4,1,1)$, then its Young diagram is
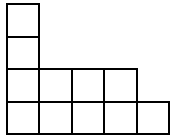
\begin{center}
$\yng(1,1,4,5)$
\end{center}

Since $\l$ is a decreasing list, reflection across the line $y=x$ gives rise to another Young diagram, called the \textbf{conjugate partition} $\mu$ of $\lambda$. For instance, relfecting the $\lambda$ above yields $\mu=(4,2,2,2,1)$:

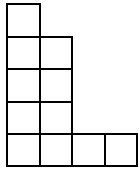
\begin{center}
\yng(1,2,2,2,4)
\end{center}

In the proof of the following proposition, it will be helpful to observe the formal definition of the conjugate partition $\mu=(\mu_1,\dots,\mu_\ell)$:
$$ \mu_m = \#\{j: \l_j\geq m \}, \qquad 1\leq m\leq\l_1=:\ell. $$
It will also be convenient to write $[N]$ as shorthand for the set $\{1,2,\dots, N\}$.

\begin{proposition}  \label{h-vec-distinct-simplices}
Let $\Gamma_\lambda$ be a disjoint union of simplices $\Delta_{1}\cup\cdots\cup\Delta_{\a}$ where $\l$ is a partition and each $\Delta_i$ is a $(\l_i-1)$-simplex. Then $h_0=1$, and for all $1\leq k\leq \l_1$
$$(-1)^{k-1} h_k = \sum_{m=1}^{\ell-k+1}\binom{\ell-m}{k-1}(\mu_m-1),$$
where $\mu$ is the conjugate partition of $\l.$
\end{proposition} 

\begin{proof}
Beginning with the definition, we apply \textbf{Proposition \ref{f-vec-distinct-simplices}}, separating out the exceptional $f_{-1}$ term from the others and then swapping the order of summation.
\begin{align*}
h_k &= \sum_{i=0}^k (-1)^{k-i}\binom{\l_1-i}{k-i} f_{i-1} \\
 &= (-1)^k\binom{\l_1}{k} + \sum_{i=1}^k (-1)^{k-i}\binom{\l_1-i}{k-i} \sum_{j=1}^{\a} \binom{\l_j}{i} \\
 &= (-1)^k\binom{\l_1}{k} + \sum_{j=1}^\a  \sum_{i=1}^{k} (-1)^{k-i}\binom{\l_1-i}{k-i}\binom{\l_j}{i}.
\end{align*}
At this point, we observe the following combinatorial fact, which we will prove later.

\begin{lemma} \label{combinatorial-technicality}
Let $a, b,$ and $k$ be non-negative integers. Then
$$ \sum_{i=0}^k (-1)^i \binom{a+b-i}{k-i}\binom{b}{i} = \binom{a}{k}. $$
\end{lemma} 

We apply this fact with $a=\l_1-\l_j$ and $b=\l_j$. When plugging in to the formula for $h_k$ above, note that we need to be careful with the $i=0$ term:
$$ h_k = (-1)^k\binom{\l_1}{k} + \sum_{j=1}^r (-1)^k\left[\binom{\l_1-\l_j}{k}-\binom{\l_1}{k} \right]. $$
Observe that when $k>0$, the $j=1$ term contributes $\binom{\l_1-\l_1}{k}-\binom{\l_1}{k} = -\binom{\l_1}{k}$. Therefore, we can cancel it with the exceptional term and start the sum at $j=2$:
\begin{align*}
h_k &= (-1)^k\sum_{j=2}^r \left[\binom{\l_1-\l_j}{k}-\binom{\l_1}{k} \right].
\end{align*}
We divide both sides by $(-1)^{k+1}$ and interpret this sum as follows:
$$  (-1)^{k+1}h_k = \sum_{j=2}^r \left[ \sum_{\substack{ S\subseteq [\l_1] \\ |S|=k}} 1 - \sum_{\substack{ S\subseteq [\l_1]\smallsetminus[\l_j] \\ |S|=k}} 1 \right]. $$
In other words, each subset $S\subseteq[\l_1]$ contributes to the inner sum precisely if it is not a subset of $\{\l_j+1,\dots, \l_1\}$. This means that $S$ contributes to the $j$ term of the outer sum precisely when its smallest element is at least $\l_j$. In other words, by swapping the order of summation, we obtain
$$ (-1)^{k+1} h_k =  \sum_{\substack{ S\subset [\l_1] \\ |S|=k}} \#\{j\geq 2: \l_j\geq \min(S) \}. $$

Now split this sum into parts by tracking the minimum element of each set (which exists because $|S|=k>0$). Once the minimum element of a $k$-element set $S$ is known to be $m$, the other elements may form any $(k-1)$-element subset of $\{m+1,\dots, \l_1\}$. Therefore:
\begin{align*}
 (-1)^{k+1} h_k &= \sum_{m=1}^{\l_1}\left[\sum_{\substack{ S\subset \{m+1,\dots,\l_1\} \\ |S|=k}} \#\{j\geq 2: \l_j\geq m \}\right] \\
 &= \sum_{m=1}^{\l_1}\binom{\l_1-m}{k-1} \#\{j\geq 2: \l_j\geq m \}.
\end{align*}
This is the desired identity, slightly disguised. First, the binomial coefficient ensures the terms vanish when $m\geq \l_1-k+1$, so we can match the upper limit of the sum. Finally, $\#\{j\geq 2: \l_j\geq m \}=\mu_m-1$, by definition of the conjugate partition and because $\l_1\geq m$ for all $m$ in the summation.
\end{proof}

For completeness, we now prove \textbf{Lemma \ref{combinatorial-technicality}}.

\begin{proof}[Proof (of Lemma \ref{combinatorial-technicality})]

The unsigned version of the left-hand side
$$ \sum_{i=0}^k  \binom{a+b-i}{k-i}\binom{b}{i}$$
counts the number of ways to choose $k$ elements from $[a+b]$ in two phrases: choosing first $i$ elements from $[a]$, and choosing second $k-i$ more (distinct) elements from $[a+b]$. There is a sign-reversing involution on such two-phase sets, given by swapping the phase in which the smallest element was chosen. This cancels the contribution of all subsets except for those with no elements in $[a]$, since it is impossible to choose any of their elements in the first phase. Therefore, each of these $\binom{a}{k}$ sets contributes exactly once to the sum, which gives the desired identity.
\end{proof}

\begin{corollary}  \label{h-vec-same-simplices}
If $\Gamma=\Delta_{1}\cup\ldots\cup\Delta_{\a}$ where each $\Delta_{i}$ is a $d$-simplex, then the $h$-vector of $\Gamma$ is given by $h_0=1$ and $h_k=(-1)^{k+1}(\a-1)\binom{d+1}{k}$ for all $1\leq k\leq d$.
\end{corollary}

We were led to \textbf{Proposition \ref{h-vec-distinct-simplices}} by observing two qualitative features of \textbf{Corollary \ref{h-vec-same-simplices}} which we tried to generalize. The first of these is that, $h_0$ notwithstanding, the signs in the $h$-vector alternate; we now have a satisfactory explanation. However, the second property remains mysterious even with the explicit formula. A sequence $(a_1,a_2,\dots)$ is called $\textbf{log-concave}$ if $a_k^2 \geq a_{k-1}a_{k+1}$ for all $k\geq 2$; this property means that the sequence increases ``smoothly'' until it hits a maximum, and then decreases smoothly afterward. Many combinatorial sequences, such as binomial coefficients and Eulerian numbers, are known to be log-concave; we conjecture that so too is the unsigned $h$-vector for general disjoint unions of simplices.

\begin{conjecture} \label{distinct-simplices-unimodal}
Let $\Gamma_\lambda$ be a disjoint union of simplices $\Delta_{1}\cup\cdots\cup\Delta_{\a}$ where $\l$ is a partition and each $\Delta_i$ is a $(\l_i-1)$-simplex. If the $d_i$ are all distinct, then the unsigned $h$-vector $( |h_1|, |h_2| ,\dots, |h_{\l_1}|)$ is log-concave.
\end{conjecture}

It is worth noting that the analogous conjecture is false for the $f$-vector; for instance, having a single high-dimensional simplex, together with an excessive number of low-dimensional ones, will cause the $f$-vector to spike in low dimension, which is not permitted in a log-concave sequence.

\subsection{Subspace Arrangements and Intersection Posets}


The $f$-vector and $h$-vector are two combinatorial features of $\Gamma$, and we now wish to discuss a geometric one. Fix a field $\mathbb{K}$. A \textbf{subspace arrangement} is a finite collection of subspaces $\mathcal A$ in $\mathbb{K}^r$ for some $r$, such that $S\not\subseteq T$ for any $S,T\in\mathcal A$. For any simplicial complex $\Gamma$ on $V=\{v_1,\dots, v_m\}$, we define the \textbf{associated subspace arrangement} $\mathcal A_\Gamma$ (or $\mathcal A$ if there is no risk of confusion), which consists of all subspaces in $\mathbb K^m$ of the form
$$ S_F = \{(z_1,\dots, z_m): z_{i}=0 \text{ for all } v_i\in V\setminus F \}, $$
where $F$ is an \emph{facet} of $\Gamma$.

This object arises naturally in algebraic geometry: the \textbf{Stanley-Reisner ideal} $I_\Gamma$ of $\Gamma$ is the ideal in the polynomial ring $\mathbb K[v_1,\dots, v_m]$ which is generated by the monomials $\prod_{v\notin F} v$ for all faces $F\in\Gamma$. One studies the geometric properties of ideals $J$ in polynomial rings by considering their \textbf{variety} $V(J)\subseteq \mathbb K^m$, defined as $V(J) = \{z\in \mathbb K^{m}: f(z)=0 \text{ for all } f\in J\}.$ It is a straightforward but tedious exercise in element-chasing to show that $V(I_\Gamma)$ is the union of all the subspaces in $\mathcal A_\Gamma$.

We now define a tool to record information about subspace arrangements. The \textbf{intersection poset} of a subspace arrangement $\mathcal A$, denoted $\inlatA$, is the finite set containing all intersections of subspaces $S_F\in\mathcal{A}$, ordered by reverse-inclusion: $I\leq J$ if $J\subseteq I$. Notice that this poset has a minimum element $\hat 0=\mathbb K^{|V|}$ given by the empty intersection, and and a maximum element $\hat 1$, given by intersecting all of the subspaces in the arrangement. Note that the subspaces $S$ in $\mathcal A$ themselves are elements of $\inlatA$. They are not minimal elements because $S\subsetneq \mathbb K^{|V|}$, but they are \textbf{atoms}; that is, for any intersection $I\in\inlatA$, it is impossible for $\hat 0 < I < S$.

A \textbf{chain} is a collection of subspaces $S_0,\dots,S_r\subseteq \inlatA$ such that $S_0\leq\cdots\leq S_r$; the number $r$ is called the \textbf{length} of the chain. A chain $\{S_1,\dots, S_r\}$ is called \textbf{maximal} if it is not a proper subset of any (longer) chain. Finally, $\inlatA$ is called \textbf{graded of rank $r$} if every maximal chain has the same length $r$. The following proposition shows that the intersection poset of the subspace arrangement of a disjoint union of simplices is graded of rank $2$.

\begin{proposition}\label{max-chain-union}
 Let $\Gamma=\Delta_1\cup\ldots\cup\Delta_\alpha,$ be a simplicial complex on a vertex set $V$ where each $\Delta_i$ is a $d_i$-simplex. Let $\inlatA$ be the corresponding intersection poset. Then every maximal chain in $\inlatA$ is of the form $C=\{\hat 0=a_0<a_1<a_2=\hat 1\}$.
\end{proposition}

\begin{proof}
Let $C=\{a_0<\dots<a_r\}$ be a maximal chain in $\inlatA.$ Since $C$ is maximal, $a_0=\hat 0$ and $a_r=\hat 1$, or else $\hat 0$ or $\hat 1$ could be added to create a longer chain. Moreover, $a_1$ must be an atom, otherwise there would be some element $a\in\inlatA$ such that $a_0<a<a_1$, but then $C'=\{a_0<a<a_1<\cdots<a_r\}$ would be a longer chain than $C$.

Let $F_1,F_2$ be facets in $\Gamma$ and let $S_1,S_2\in\mathcal{A}$ be the corresponding subspaces. Since $F_1,F_2\in\Gamma,$ then $F_1\cap F_2=\varnothing.$ Furthermore, since $F_1\cap F_2=\varnothing,$ then $(V\setminus F_1)\cup(V\setminus F_2)=V.$ Therefore, by definition of $S_1,S_2\in\mathcal{A},$ $S_1\cap S_2=\{0\}=a_2.$ 

The intersection of any subspaces in $\mathcal{A}$ is thus $\{0\},$ so any maximal chain is of the form $C=\{\hat 0=a_0<a_1<a_2=\hat 1\},$ as desired.
\end{proof}

Although in this paper all intersection posets will arise from simplicial complexes, historically the most well-studied subspace arrangements are those for which all subspaces have codimension one, known as \textbf{central hyperplane arrangements}. Central hyperplane arrangements in $\mathbb K^r$ are graded of rank $r$, essentially because all atoms have codimension one, so intersections having codimension two are precisely the intersections of two subspaces. Similarly, codimension three corresponds to the intersection of three subspaces, and so on.

For general subspace arrangements, the subspaces may have higher codimension. So the above argument fails, and indeed there no longer a guarantee that $\inlatA$ is graded at all. Despite this deficit, Athanasiadis \cite{Athanasiadis} showed that the so-called ``finite field method'', which counts the points in $\inlatA$ when $\mathbb K$ is a sufficiently large finite field, can be extended to general subspace arrangements. The classical result uses the notion of the characteristic polynomial for a hyperplane arrangement; the appropriate notion of a \textbf{characteristic polynomial} for a general subspace arrangement $\mathcal{A}$ is $$\chi_\mathcal{A}(x)=\sum_{t\in\inlatA}\mu(\Hat{0},t)x^{\dim(t)},$$ 
where the \textbf{M{\"o}bius function} $\mu$ is defined recursively by $\mu(s,s)=1$ and
$$\mu(s,t) = -\sum_{s\leq z< t} \mu(s,z).$$

Because of the simple structure for $\inlatA$ suggested by \textbf{Proposition \ref{max-chain-union}}, we can also explicitly compute the characteristic polynomial in this case:

\begin{proposition}\label{char-poly-union}
Let $\Gamma=\Delta_1\cup\ldots\cup\Delta_\alpha,$ be a simplicial complex on a vertex set $V$, where each $\Delta_i$ is a $d_i$-simplex. Let $\inlatA$ be the corresponding intersection poset. Then, 
  $$\chi_\mathcal{A}(x)=x^{|V|}-\sum_{i=1}^\alpha x^{d_i+1}+\a-1.$$
\end{proposition}

\begin{proof}
By Proposition \ref{max-chain-union}, every element $t\in\inlatA$ is either $\hat 0=\mathbb K^{|V|}$, or $\hat 1=\{0\}$, or an atom. Clearly $\mu(\hat{0},\hat{0})=1$, and $\dim(\hat 0)=|V|$. By the definition of an atom, $\hat 0\leq z< t$ implies $z=\hat 0$, so $\mu(\hat{0},a_1)=-1$. Moreover, every atom in $\inlatA$ is a subspace 
$$S_i = \{(z_{v_1},\ldots,z_{v_n}):z_v=0,\forall v\in V\setminus \Delta_i\}$$
for some facet $\Delta_i$, and so $\dim(S_F)=d_i+1$. Finally, every element $z\in\inlatA$ satisfies $\hat 0\leq z<\hat 1$, so $\mu(\hat{0},\hat{1})=-(1-\a)$, and of course $\dim(\hat 1)=0$.

Plugging this data into the defintion, we find $\chi_\mathcal{A}(x)=x^{|V|}-\sum_{i=1}^\alpha x^{d_i+1}+\a-1.$ as desired.
\end{proof}

We make a special note of the case in which each simplex has the same dimension.

\begin{corollary}\label{char-poly-union-eqdim}
If $\Gamma=\Delta_1 \cup \cdots\cup \Delta_\a$ where each $\Delta_i$ is a $\delta$-simplex, for some fixed dimension $\delta$, then
$$\chi_\mathcal{A}(x)=x^{\a(\delta+1)}-\a\cdot x^{\delta+1}+(\a-1).$$
\end{corollary}

\section{Main Results and Conjectures}

Our main results are a complete description of $\Delta_{n,\ell}$ for certain families of $n$ and $\ell$, as well as some implications for their intersection posets. In particular, it is interesting that these posets are graded for all of the families that we considered here, and we wish to stress that this is not true for every $\Delta_{n,\ell}$. In this section we only present the results, leaving the calculations themselves to Section \ref{all-the-proofs}.

\begin{restatable}{theorem}{doubling}
\label{total-doubling}
Consider $\Gamma=\Delta_{2\ell,\ell}$ for any integer $\ell$. Let $\ell=\rho\cdot 2^m$ with $m\geq 0$ and $\rho$ an odd integer, and $\mathcal A$ be the subspace arrangement associated to $\Gamma$.
\begin{enumerate}[(a)]
    \item\label{doubling} $\Gamma = \Delta_0\cup \Delta_1\cup\cdots\cup \Delta_{2^{m-1}}$, where $\Delta_t$ is the $(\rho-1)$-simplex on the set $V_t$ of all $x\in \mathbb Z/2\ell\mathbb{Z}$ congruent to $2t + 1$ mod $2^{m+1}$.
    \item $\inlatA$ is a graded poset of rank $2$.
    \item The characteristic polynomial of $\mathcal A$ is $\chi_\mathcal{A}(x)=x^{\ell}-2^mx^{\rho}+(2^m-1).$
\end{enumerate}
\end{restatable}

\begin{example}
We draw some typical examples of the above theorem: $\Delta_{2^{m+1}\rho,2^m\rho}$ for $\rho=3$ and $m=1,2,3$.
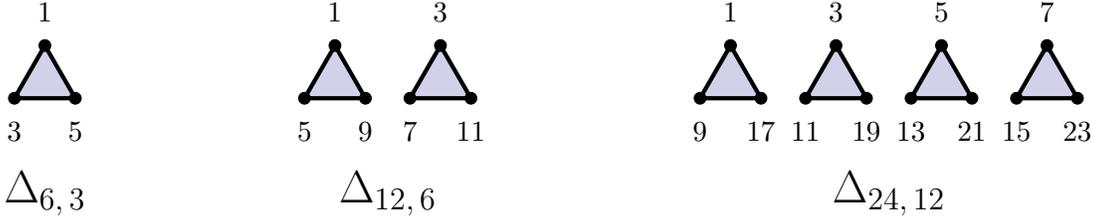
\begin{center}
\begin{tikzpicture}[scale=0.35]
\pgfmathsetlengthmacro{\rad}{38pt};
\pgfmathsetlengthmacro{\vertsize}{\rad*0.1666};
\pgfmathsetlengthmacro{\distance}{1.5*\rad};
\pgfmathsetlengthmacro{\largedist}{8*\rad};

\pgfmathsetmacro{\triangleopacity}{0.4};
\definecolor{dimtwocolor}{RGB}{140,140,200}
\begin{scope}[shift={(0,0)}]
\begin{scope}[shift={(0,0)}]
	\coordinate (a1) at (90:\rad);
	\coordinate (b1) at (210:\rad);
	\coordinate (c1) at (330:\rad);
	\draw [ultra thick, draw=black, fill=dimtwocolor, fill opacity=\triangleopacity] (a1) -- (b1) -- (c1) -- cycle;
	\draw [black, fill=black] (a1) circle [radius=\vertsize];
	\draw [black, fill=black] (b1) circle [radius=\vertsize];
	\draw [black, fill=black] (c1) circle [radius=\vertsize];
	\node[yshift=2*\vertsize] at (a1) {1};
	\node[yshift=-2*\vertsize] at (b1) {3};
	\node[yshift=-2*\vertsize] at (c1) {5};

	\node[xshift=0*\distance,yshift=-.75*\distance] at (0,0) {\LARGE $\Delta_{6,\,3}$};
\end{scope}

\begin{scope}[shift={(6.5*\distance,0)}]
	\begin{scope}[shift={(-\distance,0)}]
		\coordinate (a1) at (90:\rad);
		\coordinate (b1) at (210:\rad);
		\coordinate (c1) at (330:\rad);
		\draw [ultra thick, draw=black, fill=dimtwocolor, fill opacity=\triangleopacity] (a1) -- (b1) -- (c1) -- cycle;
		\draw [black, fill=black] (a1) circle [radius=\vertsize];
		\draw [black, fill=black] (b1) circle [radius=\vertsize];
		\draw [black, fill=black] (c1) circle [radius=\vertsize];
		\node[yshift=2*\vertsize] at (a1) {1};
		\node[yshift=-2*\vertsize] at (b1) {5};
		\node[yshift=-2*\vertsize] at (c1) {9};
	\end{scope}

	\begin{scope}[shift={(\distance,0)}]
		\coordinate (a1) at (90:\rad);
		\coordinate (b1) at (210:\rad);
		\coordinate (c1) at (330:\rad);
		\draw [ultra thick, draw=black, fill=dimtwocolor, fill opacity=\triangleopacity] (a1) -- (b1) -- (c1) -- cycle;
		\draw [black, fill=black] (a1) circle [radius=\vertsize];
		\draw [black, fill=black] (b1) circle [radius=\vertsize];
		\draw [black, fill=black] (c1) circle [radius=\vertsize];
		\node[yshift=2*\vertsize] at (a1) {3};
		\node[yshift=-2*\vertsize] at (b1) {7};
		\node[yshift=-2*\vertsize] at (c1) {11};
	\end{scope}

	\node[xshift=0*\distance,yshift=-.75*\distance] at (0,0) {\LARGE $\Delta_{12,\,6}$};
\end{scope}

\begin{scope}[shift={(16*\distance,0)}]
	\begin{scope}[shift={(-3*\distance,0)}]
		\coordinate (a1) at (90:\rad);
		\coordinate (b1) at (210:\rad);
		\coordinate (c1) at (330:\rad);
		\draw [ultra thick, draw=black, fill=dimtwocolor, fill opacity=\triangleopacity] (a1) -- (b1) -- (c1) -- cycle;
		\draw [black, fill=black] (a1) circle [radius=\vertsize];
		\draw [black, fill=black] (b1) circle [radius=\vertsize];
		\draw [black, fill=black] (c1) circle [radius=\vertsize];
		\node[yshift=2*\vertsize] at (a1) {1};
		\node[yshift=-2*\vertsize] at (b1) {9};
		\node[yshift=-2*\vertsize] at (c1) {17};
	\end{scope}

	\begin{scope}[shift={(-\distance,0)}]
		\coordinate (a1) at (90:\rad);
		\coordinate (b1) at (210:\rad);
		\coordinate (c1) at (330:\rad);
		\draw [ultra thick, draw=black, fill=dimtwocolor, fill opacity=\triangleopacity] (a1) -- (b1) -- (c1) -- cycle;
		\draw [black, fill=black] (a1) circle [radius=\vertsize];
		\draw [black, fill=black] (b1) circle [radius=\vertsize];
		\draw [black, fill=black] (c1) circle [radius=\vertsize];
		\node[yshift=2*\vertsize] at (a1) {3};
		\node[yshift=-2*\vertsize] at (b1) {11};
		\node[yshift=-2*\vertsize] at (c1) {19};
	\end{scope}

	\begin{scope}[shift={(\distance,0)}]
		\coordinate (a1) at (90:\rad);
		\coordinate (b1) at (210:\rad);
		\coordinate (c1) at (330:\rad);
		\draw [ultra thick, draw=black, fill=dimtwocolor, fill opacity=\triangleopacity] (a1) -- (b1) -- (c1) -- cycle;
		\draw [black, fill=black] (a1) circle [radius=\vertsize];
		\draw [black, fill=black] (b1) circle [radius=\vertsize];
		\draw [black, fill=black] (c1) circle [radius=\vertsize];
		\node[yshift=2*\vertsize] at (a1) {5};
		\node[yshift=-2*\vertsize] at (b1) {13};
		\node[yshift=-2*\vertsize] at (c1) {21};
	\end{scope}

	\begin{scope}[shift={(3*\distance,0)}]
		\coordinate (a1) at (90:\rad);
		\coordinate (b1) at (210:\rad);
		\coordinate (c1) at (330:\rad);
		\draw [ultra thick, draw=black, fill=dimtwocolor, fill opacity=\triangleopacity] (a1) -- (b1) -- (c1) -- cycle;
		\draw [black, fill=black] (a1) circle [radius=\vertsize];
		\draw [black, fill=black] (b1) circle [radius=\vertsize];
		\draw [black, fill=black] (c1) circle [radius=\vertsize];
		\node[yshift=2*\vertsize] at (a1) {7};
		\node[yshift=-2*\vertsize] at (b1) {15};
		\node[yshift=-2*\vertsize] at (c1) {23};
	\end{scope}

	\node[xshift=0*\distance,yshift=-.75*\distance] at (0,0) {\LARGE $\Delta_{24,\,12}$};
\end{scope}
\end{scope}
\end{tikzpicture}
\end{center}
\end{example}

In the previous section we discussed properties of complexes which are disjoint unions of simplices. The complex $\Delta_{2\ell,\ell}$ is thus a ``naturally occurring'' instance of such an object in which all the simplices have the same dimension. One may wonder whether more general unions of simplices also arise from a $\Delta_{n,\ell}$, and the next result answers this in the affirmative.

\begin{restatable}{theorem}{primepowers}
\label{total-primepowers}
Consider $\Gamma=\Delta_{p^e, p^e-1}$ for prime $p$ and and positive integer $e$. Let $\mathcal A$ be the subspace arrangement associated to $\Gamma$.
\begin{enumerate}[(a)]
    \item $\Delta_{p^{e},p^{e}-1}$ is a disjoint union of $e(p-1)$ simplices, with $e-1$ many $\left(p^j-1\right)$-simplices for each $0\leq j\leq e-1$. 
    \item $\inlatA$ is a graded poset of rank $2$.
    \item The characteristic polynomial of $\mathcal A$ is $\chi_\mathcal{A}(x)=x^{p^e-1}-(p-1)\sum^{e-1}_{j=0}x^{p^j}$.
\end{enumerate}
\end{restatable}

\begin{example}
We draw a typical example of the above theorem: $\Delta_{p^e,p^e-1}$ for $p=3,e=2$. 
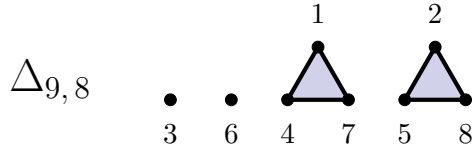
\begin{center}
\begin{tikzpicture}[scale=0.35]
\pgfmathsetlengthmacro{\rad}{38pt};
\pgfmathsetlengthmacro{\vertsize}{\rad*0.1666};
\pgfmathsetlengthmacro{\distance}{1.5*\rad};
\pgfmathsetlengthmacro{\largedist}{8*\rad};

\pgfmathsetmacro{\triangleopacity}{0.4};
\definecolor{dimtwocolor}{RGB}{140,140,200}

\begin{scope}[shift={(0,0)}]
    \begin{scope}[shift={(0*\vertsize,0)}]
	\coordinate (d1) at (210:\rad);
	\coordinate (e1) at (330:\rad);
	\draw [black, fill=black] (d1) circle [radius=\vertsize];
	\draw [black, fill=black] (e1) circle [radius=\vertsize];
	\node[yshift=-2*\vertsize] at (d1) {3};
	\node[yshift=-2*\vertsize] at (e1) {6};
    \end{scope}
	\begin{scope}[shift={(20*\vertsize,0)}]
		\coordinate (a1) at (90:\rad);
		\coordinate (b1) at (210:\rad);
		\coordinate (c1) at (330:\rad);
		\draw [ultra thick, draw=black, fill=dimtwocolor, fill opacity=\triangleopacity] (a1) -- (b1) -- (c1) -- cycle;
		\draw [black, fill=black] (a1) circle [radius=\vertsize];
		\draw [black, fill=black] (b1) circle [radius=\vertsize];
		\draw [black, fill=black] (c1) circle [radius=\vertsize];
		\node[yshift=2*\vertsize] at (a1) {1};
		\node[yshift=-2*\vertsize] at (b1) {4};
		\node[yshift=-2*\vertsize] at (c1) {7};
	\end{scope}

	\begin{scope}[shift={(40*\vertsize,0)}]
		\coordinate (a1) at (90:\rad);
		\coordinate (b1) at (210:\rad);
		\coordinate (c1) at (330:\rad);
		\draw [ultra thick, draw=black, fill=dimtwocolor, fill opacity=\triangleopacity] (a1) -- (b1) -- (c1) -- cycle;
		\draw [black, fill=black] (a1) circle [radius=\vertsize];
		\draw [black, fill=black] (b1) circle [radius=\vertsize];
		\draw [black, fill=black] (c1) circle [radius=\vertsize];
		\node[yshift=2*\vertsize] at (a1) {2};
		\node[yshift=-2*\vertsize] at (b1) {5};
		\node[yshift=-2*\vertsize] at (c1) {8};
	\end{scope}
	\node[xshift=-1*\distance,yshift=0*\distance] at (0,0) {\LARGE $\Delta_{9,\,8}$};
\end{scope}
\end{tikzpicture}
\end{center}
\end{example}

We note that these two families are not the only parameters $n$ and $\ell$ which yield a disjoint union of simplices; for instance, $\Delta_{12,9}$ is the disjoint union of a 5-simplex and a 2-simplex. It may be interesting to seek a complete classification of the parameters $n$ and $\ell$ which exhibit this phenomenon.

We also note that in \textbf{Theorem \ref{total-primepowers}}, $n-\ell=1$. In all the explicit calculations we were able to carry out, we observed that $\Delta_{n,\ell}$ is somehow ``simpler'' when $n-\ell$ is small. The last theorem may be interpreted as some further evidence for this observation.

\begin{restatable}{theorem}{armslegs}
\label{total-armslegs}
Consider $\Gamma_s = \Delta_{2p,2p-s}$ for prime $p\neq 2$ and positive integer $s\in\{1,2,3\}$. Let $\mathcal A_s$ be the subspace arrangement associated to $\Gamma_s$.
\begin{enumerate}[(a)]
    \item 
    \begin{itemize}
        \item\label{2p,2p-2} $\Gamma_2$ is a disjoint union of the $p-1$ edges $\{i_1,i_2\}$ for which $i_1\equiv{i_2}\not\equiv{0}\bmod{p}$. 
        \item\label{2p,2p-1} The facets of $\Gamma_1$ are the facets of $\Gamma_2$ together with $\{1,3,5,...,2p-1\}$.
        \item\label{2p,2p-3} If $p\geq 5$, the facets of $\Gamma_3$ are the facets of $\Gamma_1$ together with the $p-1$ edges $\{i,j\}$ for which $i$ is odd, $j\neq 0$ is even, and $j\equiv-2i\bmod{2p}$.
    \end{itemize}
    
    \item $\mathcal L(\mathcal A_s)$ is a graded poset. In particular, $\mathcal L(\mathcal A_2)$ is of rank $2$; and $\mathcal L(\mathcal A_1)$ and  $\mathcal L(\mathcal A_3)$ are each of rank $3$.
    
    \item The characteristic polynomial of $\mathcal A_s$ is
    $$ \chi_{\mathcal A_s}(x) = 
        \left\{\begin{array}{lr}
        x^{2p-1}-x^{p}-(p-1)x^2+(p-1)x & : s=1\\
        x^{2p-2}-p\cdot x^{2}+p-1 & : s=2\\
        x^{2p-1}-x^p-2(p-1)x^2+2(p-1)x & : s=3
        \end{array}\right.
    $$
\end{enumerate}
\end{restatable}

\begin{example}

We draw some typical examples of the above theorem: $\Delta_{2p,2p-s}$ for $p=7$ and $s=1,2,3$.

\begin{center}
\begin{tikzpicture}[scale=0.25]
\pgfmathsetlengthmacro{\rad}{70pt};
\pgfmathsetlengthmacro{\vertsize}{\rad*0.1666};
\pgfmathsetlengthmacro{\distance}{1.5*\rad};
\pgfmathsetlengthmacro{\largedist}{5.5*\rad};

\pgfmathsetmacro{\p}{7};
\pgfmathsetmacro{\pp}{\p-1};
\pgfmathsetmacro{\angleSep}{360/\pp};
\pgfmathsetmacro{\angleOff}{\angleSep*0.5};

\pgfmathsetmacro{\triangleopacity}{0.4};
\pgfmathsetmacro{\subfaceopacity}{0.3};
\definecolor{dimtwocolor}{RGB}{140,140,200};
\definecolor{dimhighcolor}{RGB}{70,70,100};
\begin{scope}[shift={(0,0)}]
\begin{scope}[shift={(0*\distance,0)}]
	\coordinate (C) at (\vertsize,1.5*\vertsize);

	\coordinate (I1) at (\angleOff-0*\angleSep:\rad);
	\coordinate (I2) at (\angleOff-1*\angleSep:\rad);
	\coordinate (I3) at (\angleOff-2*\angleSep:\rad);
	\coordinate (I4) at (\angleOff-3*\angleSep:\rad);
	\coordinate (I5) at (\angleOff-4*\angleSep:\rad);
	\coordinate (I6) at (\angleOff-5*\angleSep:\rad);
	\coordinate (O1) at (\angleOff-0*\angleSep:2*\rad);
	\coordinate (O2) at (\angleOff-1*\angleSep:2*\rad);
	\coordinate (O3) at (\angleOff-2*\angleSep:2*\rad);
	\coordinate (O4) at (\angleOff-3*\angleSep:2*\rad);
	\coordinate (O5) at (\angleOff-4*\angleSep:2*\rad);
	\coordinate (O6) at (\angleOff-5*\angleSep:2*\rad);

	\draw [ultra thick, draw=black, fill=dimhighcolor, fill opacity=2*\triangleopacity] (I1) -- (I2) -- (I3) -- (I4) -- (I5) -- (I6)-- cycle;
	\draw [ultra thick, draw=black, opacity=\subfaceopacity] (I1) -- (I3);
	\draw [ultra thick, draw=black, opacity=\subfaceopacity] (I1) -- (I4);
	\draw [ultra thick, draw=black, opacity=\subfaceopacity] (I1) -- (I5);
	\draw [ultra thick, draw=black, opacity=\subfaceopacity] (I1) -- (I6);
	\draw [ultra thick, draw=black, opacity=\subfaceopacity] (I2) -- (I4);
	\draw [ultra thick, draw=black, opacity=\subfaceopacity] (I2) -- (I5);
	\draw [ultra thick, draw=black, opacity=\subfaceopacity] (I2) -- (I6);
	\draw [ultra thick, draw=black, opacity=\subfaceopacity] (I3) -- (I5);
	\draw [ultra thick, draw=black, opacity=\subfaceopacity] (I3) -- (I6);
	\draw [ultra thick, draw=black, opacity=\subfaceopacity] (I4) -- (I6);
	\draw [ultra thick, draw=black, opacity=2*\subfaceopacity] (C) -- (I1);
	\draw [ultra thick, draw=black, opacity=2*\subfaceopacity] (C) -- (I2);
	\draw [ultra thick, draw=black, opacity=2*\subfaceopacity] (C) -- (I3);
	\draw [ultra thick, draw=black, opacity=2*\subfaceopacity] (C) -- (I4);
	\draw [ultra thick, draw=black, opacity=2*\subfaceopacity] (C) -- (I5);
	\draw [ultra thick, draw=black, opacity=2*\subfaceopacity] (C) -- (I6);

	\draw [ultra thick, draw=black] (I1) -- (O1);
	\draw [ultra thick, draw=black] (I2) -- (O2);
	\draw [ultra thick, draw=black] (I3) -- (O3);
	\draw [ultra thick, draw=black] (I4) -- (O4);
	\draw [ultra thick, draw=black] (I5) -- (O5);
	\draw [ultra thick, draw=black] (I6) -- (O6);

	\draw [black, fill=black] (C) circle [radius=\vertsize];
	\draw [black, fill=black] (I1) circle [radius=\vertsize];
	\draw [black, fill=black] (I2) circle [radius=\vertsize];
	\draw [black, fill=black] (I3) circle [radius=\vertsize];
	\draw [black, fill=black] (I4) circle [radius=\vertsize];
	\draw [black, fill=black] (I5) circle [radius=\vertsize];
	\draw [black, fill=black] (I6) circle [radius=\vertsize];
	\draw [black, fill=black] (O1) circle [radius=\vertsize];
	\draw [black, fill=black] (O2) circle [radius=\vertsize];
	\draw [black, fill=black] (O3) circle [radius=\vertsize];
	\draw [black, fill=black] (O4) circle [radius=\vertsize];
	\draw [black, fill=black] (O5) circle [radius=\vertsize];
	\draw [black, fill=black] (O6) circle [radius=\vertsize];

	\node[xshift=0.75*\vertsize, yshift=-0.5*\vertsize] at (I1) {3};
	\node[yshift=-0.95*\vertsize] at (I2) {9};
	\node[xshift=-0.90*\vertsize, yshift=-0.40*\vertsize] at (I3) {13};
	\node[xshift=-0.75*\vertsize, yshift=0.65*\vertsize] at (I4) {11};
	\node[yshift=1*\vertsize] at (I5) {5};
	\node[xshift=0.75*\vertsize, yshift=0.5*\vertsize] at (I6) {1};

	\node[xshift=1.25*\vertsize] at (O1) {10};
	\node[xshift=0.9*\vertsize, yshift=-0.9*\vertsize] at (O2) {2};
	\node[yshift=-1.25*\vertsize] at (O3) {6};
	\node[xshift=-1.25*\vertsize] at (O4) {4};
	\node[xshift=-0.9*\vertsize, yshift=0.9*\vertsize] at (O5) {12};
	\node[yshift=1.15*\vertsize] at (O6) {8};

	\node[xshift=0*\distance,yshift=-.75*\distance] {\LARGE $\Delta_{14,\,13}$};
\end{scope}

\begin{scope}[shift={(6*\distance,0)}]
	\coordinate (C) at (\vertsize,1.5*\vertsize);

	\coordinate (T1) at (-2.5*\rad,0.5*\rad);
	\coordinate (T2) at (-1.5*\rad,0.5*\rad);
	\coordinate (T3) at (-0.5*\rad,0.5*\rad);
	\coordinate (T4) at (2.5*\rad,0.5*\rad);
	\coordinate (T5) at (1.5*\rad,0.5*\rad);
	\coordinate (T6) at (0.5*\rad,0.5*\rad);
	\coordinate (B1) at (-2.5*\rad,-0.5*\rad);
	\coordinate (B2) at (-1.5*\rad,-0.5*\rad);
	\coordinate (B3) at (-0.5*\rad,-0.5*\rad);
	\coordinate (B4) at (2.5*\rad,-0.5*\rad);
	\coordinate (B5) at (1.5*\rad,-0.5*\rad);
	\coordinate (B6) at (0.5*\rad,-0.5*\rad);

	\coordinate (O1) at (-0*\angleSep:2*\rad);
	\coordinate (O2) at (-1*\angleSep:2*\rad);
	\coordinate (O3) at (-2*\angleSep:2*\rad);
	\coordinate (O4) at (-3*\angleSep:2*\rad);
	\coordinate (O5) at (-4*\angleSep:2*\rad);
	\coordinate (O6) at (-5*\angleSep:2*\rad);

	\draw [ultra thick, draw=black] (T1) -- (B1);
	\draw [ultra thick, draw=black] (T2) -- (B2);
	\draw [ultra thick, draw=black] (T3) -- (B3);
	\draw [ultra thick, draw=black] (T4) -- (B4);
	\draw [ultra thick, draw=black] (T5) -- (B5);
	\draw [ultra thick, draw=black] (T6) -- (B6);

	\draw [black, fill=black] (T1) circle [radius=\vertsize];
	\draw [black, fill=black] (T2) circle [radius=\vertsize];
	\draw [black, fill=black] (T3) circle [radius=\vertsize];
	\draw [black, fill=black] (T4) circle [radius=\vertsize];
	\draw [black, fill=black] (T5) circle [radius=\vertsize];
	\draw [black, fill=black] (T6) circle [radius=\vertsize];
	\draw [black, fill=black] (B1) circle [radius=\vertsize];
	\draw [black, fill=black] (B2) circle [radius=\vertsize];
	\draw [black, fill=black] (B3) circle [radius=\vertsize];
	\draw [black, fill=black] (B4) circle [radius=\vertsize];
	\draw [black, fill=black] (B5) circle [radius=\vertsize];
	\draw [black, fill=black] (B6) circle [radius=\vertsize];

	\node[yshift=1*\vertsize] at (T1) {1};
	\node[yshift=1*\vertsize] at (T2) {3};
	\node[yshift=1*\vertsize] at (T3) {5};
	\node[yshift=1*\vertsize] at (T4) {9};
	\node[yshift=1*\vertsize] at (T5) {11};
	\node[yshift=1*\vertsize] at (T6) {13};
	\node[yshift=-1*\vertsize] at (B1) {8};
	\node[yshift=-1*\vertsize] at (B2) {10};
	\node[yshift=-1*\vertsize] at (B3) {12};
	\node[yshift=-1*\vertsize] at (B4) {2};
	\node[yshift=-1*\vertsize] at (B5) {4};
	\node[yshift=-1*\vertsize] at (B6) {6};

	\node[xshift=0*\distance,yshift=-.75*\distance] {\LARGE $\Delta_{14,\,12}$};
\end{scope}

\begin{scope}[shift={(12*\distance,0)}]
	\coordinate (C) at (\vertsize,1.5*\vertsize);

	\coordinate (I1) at (\angleOff-0*\angleSep:\rad);
	\coordinate (I2) at (\angleOff-1*\angleSep:\rad);
	\coordinate (I3) at (\angleOff-2*\angleSep:\rad);
	\coordinate (I4) at (\angleOff-3*\angleSep:\rad);
	\coordinate (I5) at (\angleOff-4*\angleSep:\rad);
	\coordinate (I6) at (\angleOff-5*\angleSep:\rad);
	\coordinate (O1) at (-0*\angleSep:2*\rad);
	\coordinate (O2) at (-1*\angleSep:2*\rad);
	\coordinate (O3) at (-2*\angleSep:2*\rad);
	\coordinate (O4) at (-3*\angleSep:2*\rad);
	\coordinate (O5) at (-4*\angleSep:2*\rad);
	\coordinate (O6) at (-5*\angleSep:2*\rad);

	\draw [ultra thick, draw=black, fill=dimhighcolor, fill opacity=2*\triangleopacity] (I1) -- (I2) -- (I3) -- (I4) -- (I5) -- (I6)-- cycle;
	\draw [ultra thick, draw=black, opacity=\subfaceopacity] (I1) -- (I3);
	\draw [ultra thick, draw=black, opacity=\subfaceopacity] (I1) -- (I4);
	\draw [ultra thick, draw=black, opacity=\subfaceopacity] (I1) -- (I5);
	\draw [ultra thick, draw=black, opacity=\subfaceopacity] (I1) -- (I6);
	\draw [ultra thick, draw=black, opacity=\subfaceopacity] (I2) -- (I4);
	\draw [ultra thick, draw=black, opacity=\subfaceopacity] (I2) -- (I5);
	\draw [ultra thick, draw=black, opacity=\subfaceopacity] (I2) -- (I6);
	\draw [ultra thick, draw=black, opacity=\subfaceopacity] (I3) -- (I5);
	\draw [ultra thick, draw=black, opacity=\subfaceopacity] (I3) -- (I6);
	\draw [ultra thick, draw=black, opacity=\subfaceopacity] (I4) -- (I6);
	\draw [ultra thick, draw=black, opacity=2*\subfaceopacity] (C) -- (I1);
	\draw [ultra thick, draw=black, opacity=2*\subfaceopacity] (C) -- (I2);
	\draw [ultra thick, draw=black, opacity=2*\subfaceopacity] (C) -- (I3);
	\draw [ultra thick, draw=black, opacity=2*\subfaceopacity] (C) -- (I4);
	\draw [ultra thick, draw=black, opacity=2*\subfaceopacity] (C) -- (I5);
	\draw [ultra thick, draw=black, opacity=2*\subfaceopacity] (C) -- (I6);

	\draw [ultra thick, draw=black] (I1) -- (O1);
	\draw [ultra thick, draw=black] (I2) -- (O2);
	\draw [ultra thick, draw=black] (I3) -- (O3);
	\draw [ultra thick, draw=black] (I4) -- (O4);
	\draw [ultra thick, draw=black] (I5) -- (O5);
	\draw [ultra thick, draw=black] (I6) -- (O6);
	\draw [ultra thick, draw=black] (I1) -- (O6);
	\draw [ultra thick, draw=black] (I2) -- (O1);
	\draw [ultra thick, draw=black] (I3) -- (O2);
	\draw [ultra thick, draw=black] (I4) -- (O3);
	\draw [ultra thick, draw=black] (I5) -- (O4);
	\draw [ultra thick, draw=black] (I6) -- (O5);

	\draw [black, fill=black] (C) circle [radius=\vertsize];
	\draw [black, fill=black] (I1) circle [radius=\vertsize];
	\draw [black, fill=black] (I2) circle [radius=\vertsize];
	\draw [black, fill=black] (I3) circle [radius=\vertsize];
	\draw [black, fill=black] (I4) circle [radius=\vertsize];
	\draw [black, fill=black] (I5) circle [radius=\vertsize];
	\draw [black, fill=black] (I6) circle [radius=\vertsize];
	\draw [black, fill=black] (O1) circle [radius=\vertsize];
	\draw [black, fill=black] (O2) circle [radius=\vertsize];
	\draw [black, fill=black] (O3) circle [radius=\vertsize];
	\draw [black, fill=black] (O4) circle [radius=\vertsize];
	\draw [black, fill=black] (O5) circle [radius=\vertsize];
	\draw [black, fill=black] (O6) circle [radius=\vertsize];

	\node[xshift=0.80*\vertsize, yshift=0.5*\vertsize] at (I1) {3};
	\node[xshift=0.8*\vertsize, yshift=-0.5*\vertsize] at (I2) {9};
	\node[yshift=-1.15*\vertsize] at (I3) {13};
	\node[xshift=-1*\vertsize, yshift=-0.5*\vertsize] at (I4) {11};
	\node[xshift=-0.9*\vertsize, yshift=0.6*\vertsize] at (I5) {5};
	\node[yshift=1*\vertsize] at (I6) {1};

	\node[xshift=1.25*\vertsize] at (O1) {10};
	\node[xshift=0.9*\vertsize, yshift=-0.9*\vertsize] at (O2) {2};
	\node[xshift=-0.9*\vertsize, yshift=-0.9*\vertsize] at (O3) {6};
	\node[xshift=-1.05*\vertsize] at (O4) {4};
	\node[xshift=-1*\vertsize, yshift=1*\vertsize] at (O5) {12};
	\node[xshift=0.9*\vertsize, yshift=0.9*\vertsize] at (O6) {8};

	\node[xshift=0*\distance,yshift=-.75*\distance] {\LARGE $\Delta_{14,\,11}$};
\end{scope}
\end{scope}
\end{tikzpicture}
\end{center}
\end{example}


We conclude this section with some avenues for further work. These conjectures were generated by explicitly computing $\Delta_{n,\ell}$ for all $n\leq 19$, using \textbf{Theorem \ref{algorithm}}. The following conjecture partially extends the investigation suggested by \textbf{Theorem \ref{total-armslegs}} to smaller $\ell$:

\begin{conjecture}
For any prime $p$ and even $\ell$ with $p>\ell\geq\frac{p-1}{2}$, the complexes $\Delta_{2p,\ell}$ have no isolated vertices; i.e., they have no facets of dimension 0.
\end{conjecture}

It is known, for instance by \cite[Theorem F.6]{menu-research-problems}, that all of these complexes are graphs. Therefore, this conjecture shows that all of their facets have dimension one. In general, a complex whose facets all have the same dimension is called \textbf{pure}. 

\begin{conjecture}
For odd $n$ each of the complexes $\Delta_{n,\frac{n-1}{2}}$ and $\Delta_{n,\frac{n+1}{2}}$ is pure if and only if $n$ is prime.
\end{conjecture}

Surprisingly, the $h$-vector---which in principle records nothing at all about facets---seems to contain a sufficient condition for purity for $\simpc$:

\begin{conjecture}
The complex $\simpc$ is pure if $h_i\geq 0$ for all $i=1,\ldots ,d$.
\end{conjecture}

In a different direction, a complex $\Gamma$ \textbf{connected} if for any two vertices $v$ and $w$, there is a sequence of vertices $(p_0=v,p_1,p_2,\dots p_{k-1}, w=p_{k})$ that forms a path from $v$ to $w$; i.e. such that $p_0=v$,$p_{k}=w$, and $\{p_i,p_{i+1}\}\in\Gamma$ for every $0\leq i\leq k-1$. We conjecture that a significant portion of all $\Delta_{n,\ell}$ are connected:

\begin{conjecture}
For any $n>2\ell$, the complex $\simpc$ is connected.
\end{conjecture}

Note that \textbf{Theorem \ref{total-doubling}} shows that the inequality is sharp, in the sense that $\Delta_{2\ell,\ell}$ is connected if and only if $\ell$ is odd. However, if $n$ is odd, the data suggests that there may be some weaker bound on $\ell$; in particular, it appears that $\Delta_{n,\frac{n+1}{2}}$ is connected for $n\geq 7$.

The topologically sophisticated reader will be aware that a complex is connected if and only if its zeroth homology group vanishes. The small examples we computed suggest that higher homology groups also vanish when $n\gg h$, but the evidence is too weak to give a more quantitative estimate.

\section{Proofs of Theorems}
\label{all-the-proofs}

Recall that an integer $N$ is called even if there is another integer $M$ such that $N=2M$, and odd otherwise. In the proofs below, we will be performing modular arithmetic but it will be helpful to have these words available. For a prime $p$, we say that $N\in\mathbb{Z}/n\mathbb{Z}$ is \textbf{divisible by $p$} if $N=pM$ for some $M\in\mathbb{Z}/n\mathbb{Z}$. If $N$ is divisible by $2$, then it is called \textbf{even}, and if not, it is called \textbf{odd}. If $p$ does not divide $n$ (as an integer), then no elements divisible by $p$, but we will only be concerned with the case when $p|n$, where these definitions are more intuitive.

We say that two vertices $v$ and $w$ are \textbf{adjacent} in a simplicial complex $\Gamma$ if $\{v,w\}\in\Gamma$. Thus, when $\Gamma=\Delta_{n,\ell}$, the vertices $v$ and $w$ are adjacent if and only if there is no solution $0\leq r\leq \ell$ to the following \textbf{critical equivalence}:
$$ rv + (\ell-r)w \equiv 0 \bmod{n}. $$ 
Note that this is not a purely number-theoretic condition: it is possible that the only solutions to the critical equivalence fall outside the range $[0,\ell]$. We call such solutions \textbf{invalid}. However, notice that if $\ell=n-1$, then validity means that $r$ must fall inside the range $[0,n-1]$, so any solution to the equivalence is equivalent modulo $n$ to a valid solution.

When checking for solutions, instead of solving the critical equivalence mod $n$ we usually will observe that the left-hand side has some factor in common with $n$, say $g$, and we will divide through everything by $g$. In general, we may write that if we have a solution $r$ to an equivalence modulo $n/g$, e.g. $r \equiv f \bmod{n/g}$, this means that $r=f+\beta(n/g)$ for some integer $\beta$. Therefore, $gr=gf+\beta n$, and so there is still a solution to the equivalence $gr\equiv gf \bmod{n}$.

Moreover, since $r$ is a solution to the equivalence modulo $n/g$, there is at least one solution to the critical equivalence (modulo $n$) in the much smaller interval $[0,n/g-1]$; and in our proofs this will usually give a solution in $[0,\ell]$.

\subsection*{Proof of Theorem \ref{total-doubling}}

\doubling*


First, note that the vertex set $V$ of $\Delta_{2\ell,\ell}$ contains precisely the odd elements of $\mathbb{Z}/2\ell\mathbb{Z}$, since if $b=2c$ is even, then $b+\cdots+b=\ell(2c)\equiv 0 \bmod{2\ell}$. Therefore, every $v\in V$ is in some $V_t$.

For part (a), begin by observing that for any $t$ and any choice of $\ell$ vertices from $V_{t}$:
$$2^{m + 1}k_1 +2t+1,\quad 2^{m + 1}k_2 +2t+1, \quad\cdots,\quad 2^{m+1}k_{\ell} +2t+1$$
(where the $k_i$ are not necessarily distinct), their sum is
$$
\sum_{i = 1}^{\ell}(2^{m+1}k_i + 2t+1) = 
\left(\sum_{i=1}^{\ell}2^{m+1}k_i\right) + (2t+1)\ell.
$$
Now suppose for the sake of contradiction that the above sum is equivalent to zero modulo $2\ell.$ Therefore, since $\ell=2^m\rho$,
\begin{align*}
(2t+1)\ell + \sum_{i=1}^{\ell}2^{m+1}k_i &\equiv 0 &&\bmod{2\ell} \\
2^m\left( (2t+1)\rho + \sum_{i=1}^{\ell}2k_i \right) &\equiv 0 &&\bmod{2\ell} \\
(2t+1)\rho +\sum_{i=1}^{\ell}2k_i &\equiv 0 &&\bmod{2\rho}.
\end{align*}
Since $\rho$ is odd, the left-hand side is odd, but this clearly contradicts that zero is even.
implying that no sum of elements in $V_t$ evaluates to zero modulo $2\ell$, proving that $V_t$ is  a $(\rho-1)$-dimensional face of $\Delta_{2\ell,\ell}$.

It thus remains to show that if $v\in V_{t_1}$ and $w\in V_{t_2}$ for $t_1\neq t_2$ then $v$ and $w$ are not adjacent. That is, we need to find a valid solution $r$ to the critical equivalence:
\begin{align*}r\cdot(2^{m+1}k_1+2t_1 + 1) + (\rho\cdot{2^m}-r)\cdot({2^{m+1}}k_2+ 2t_2+1) &\equiv 0 \bmod{n}.
\end{align*}
Routine algebraic manipulation on this equivalence yields
\begin{align*}
 2r(t_1 - t_2 + 2^{m}(k_1-k_2)) + \rho\cdot{2^{m}} &\equiv 0&&\bmod{\rho\cdot2^{m+1}} \\
r(t_1 - t_2 + 2^{m}(k_1-k_2)) &\equiv -\rho\cdot{2^{m-1}}&&\bmod{\rho\cdot2^{m}} \\
r(\tau + 2^{m-e}(k_1-k_2)) &\equiv{-\rho\cdot{2^{m-1-e}}}&&\bmod{\rho\cdot2^{m-e}},
\end{align*}
where $t_1-t_2 = \tau 2^e$ for some odd number $\tau$ and $0\leq{e}\leq{m-2}$. 

Since $\tau + 2^{m-e}(k_1-k_2)$ is odd, we have
$$\gcd(\tau + 2^{m-e}(k_1-k_2), \rho\cdot2^{m-e}) = \gcd(\tau + 2^{m-e}(k_1-k_2),\rho)$$
Therefore, letting $g=\gcd(\tau + 2^{m-e}(k_1-k_2),\rho)$, we have that

$$s=\frac{\tau + 2^{m-e}(k_1-k_2)}{g} \quad\text{and}\quad \rho'=\frac{\rho}{g}$$
are both odd integers with $\gcd(s,\rho') = 1$. Since $s$ is odd, we thus conclude that $s$ is invertible modulo $\rho'\cdot{2^{m-e}}$. Thus, continuing to simplify the critical equivalence, we find
\begin{align*}
r(gs)&\equiv{-g\rho'\cdot{2^{m-1-e}}}&&\bmod{g\rho'\cdot{2^{m-e}}}\\
rs&\equiv-\rho'2^{m-1-e}&&\bmod{\rho'\cdot2^{m-e}}\\
r&\equiv \frac{-\rho'2^{m-1-e}}{s}&&\bmod{\rho'\cdot2^{m-e}}.
\end{align*}
and thus the critical equivalence has a solution modulo $\rho'2^{m-e}$ Thus, there is a solution $r$ to the critical equivalence, and it may be chosen such that $0\leq r\leq \rho' 2^{m-e}\leq \rho 2^m\leq  \ell$.

We have thus shown that any two vertices $v\in V_{t_1}$ and $w\in V_{t_2}$ are not adjacent for any $t_1$ and $t_2$, and hence that $\Gamma=\Delta_0\cup\cdots\cup\Delta_{2^{m-1}}$, concluding the proof of part (a).

For part (b), let $\inlatA$ be the corresponding intersection poset of subspaces. Since $\Gamma$ on a vertex set $V$ is a collection of $2^m$ many $(\rho-1)$-simplices, then by \textbf{Proposition \ref{max-chain-union}}, $\inlatA$ is graded of rank $2$. 

Finally, for part (c) notice that $|V|=\ell$ and each facet has dimension $d_i=\rho-1$. So, by \textbf{Corollary \ref{char-poly-union-eqdim}}, we have $\chi_\mathcal{A}(x)=x^\ell-2^mx^\rho+(2^m-1).$


\begin{corollary}
If $n=\rho\cdot 2^{m+1}, \ell=\rho\cdot 2^{m}$ for all $m\geq{0}$ and for some odd number $\rho$, then $\simpc$ is a pure simplicial complex.
\end{corollary}

\begin{proof}
By \textbf{Theorem \ref{total-doubling}(\ref{doubling})}, $\Delta_{2^{m+1}\rho,2^m\rho}$ contains $2^m\rho$ disjoint $(\rho-1)$-simplices. Since every simplex is pure, $\Delta_{2^{m+1}\rho,2^m\rho}$ is pure.
\end{proof}

\subsection*{Proof of Theorem \ref{total-primepowers}}

\primepowers*


\begin{proof}
For part (a), begin by partitioning the vertex set of $\Delta_{p^e,p^e-1}$ into $e\cdot(p-1)$ disjoint sets, denoted $V_{i,j}$ for $1\leq i\leq p-1$ and $1\leq j\leq e$, where 
    $$V_{i,j} = \left\{x\in{V}: x\equiv i\cdot{p^{j-1}}\bmod{p^{j}},1\leq{i}\leq{p-1}\right\}.$$
We can more easily understand the structure of $\Delta_{p^e,p^e-1}$ by using the following table:
    \begin{center}
    \begin{tabular}{c|c|c|c|c|c|c}
    $\bm{\bmod{\hspace{0.1cm}p^1}}$ & $p^{0}\cdot1$&$p^0\cdot{2}$&$p^0\cdot{3}$&\ldots&$p^0\cdot{(p-2)}$&$p^0\cdot{(p-1)}$\\\hline
    $\bm{\bmod{\hspace{0.1cm}p^2}}$& $p^1\cdot{1}$&$p^1\cdot{2}$&$p^1\cdot{3}$&\ldots&$p^1\cdot{(p-2)}$&$p^1\cdot{(p-1)}$\\\hline
    $\bm{\bmod{\hspace{0.1cm}p^3}}$& $p^2\cdot{1}$&$p^2\cdot{2}$&$p^2\cdot{3}$&\ldots&$p^2\cdot{(p-2)}$&$p^2\cdot{(p-1)}$\\\hline
    $\vdots$&$\vdots$&$\vdots$&$\vdots$&$\ddots$&$\vdots$&$\vdots$\\\hline
    $\bm{\bmod{\hspace{0.1cm}p^{e-1}}}$& $p^{e-2}\cdot{1}$&$p^{e-2}\cdot{2}$&$p^{e-2}\cdot{3}$&\ldots&$p^{e-2}\cdot{(p-2)}$&$p^{e-2}\cdot{(p-1)}$\\\hline
    $\bm{\bmod{\hspace{0.1cm}p^{e}}}$& $p^{e-1}\cdot{1}$&$p^{e-1}\cdot{2}$&$p^{e-1}\cdot{3}$&\ldots&$p^{e-1}\cdot{(p-2)}$&$p^{e-1}\cdot{(p-1)}$
    \end{tabular}
    \end{center}
Note that each above cell represents the set of all vertices in the vertex set of $\Delta_{p^e,p^e-1}$ that are equivalent to the quantity labelling the cell modulo the bolded value labelling the corresponding row. For example, in the first row, the set contained in the third cell labelled ``$p^0\cdot{2}$'' is the set $\left\{x\in{V}: x\equiv{p^0\cdot{2}}\bmod{p^1}\right\}$ where $V$ is the vertex set of $\Delta_{p^e,p^e-1}$.

To prove that each $V_{i,j}$ is a simplex, observe that for any choice of $\ell=p^e-1$  vertices from $V_{i,j}$:
    $$ t_1p^j +i\cdot{p^{j-1}}, \quad  t_2p^{j} + i\cdot{p^{j-1}}, \quad\cdots\quad t_\ell p^j+i\cdot{p^{j-1}},$$ 
(where the $t_i$ are not necessarily distinct) their sum is
\begin{align*}
    \sum_{k=0}^{\ell}\left(t_kp^j + p^{j-1}\cdot{i}\right) 
    &=\left(p^e-1\right)p^{j-1}\cdot{i} + p^j\sum_{k=0}^{\ell}t_k \\
    &\equiv -p^{j-1}\cdot{i} + p^j\sum_{k=0}^{\ell}t_k&&\bmod{p^e} \\
    & \equiv -i+p\sum_{k=0}^{\ell}t_k &&\bmod{p^{e-(j-1)}}.
\end{align*}
Since $-i$ is not divisible by $p$, this sum is nonzero. Thus no sum of $p^e-1$ elements is zero modulo $n=p^e$; that is, $V_{i,j}$ is a $(p^{e-j}-1)$-dimensional face of $\Delta_{p^e,p^e-1}$.
 
It thus remains to show that if $v\in V_{i,j}$ and $w\in V_{i',j'}$ then $v$ and $w$ are not adjacent for any $i\neq i'$ or $j\neq j'$. That is, we need to find a solution $r$ to the critical equivalence:
    $$r\cdot(i + pt_1)p^{j-1} + (p^e - 1 - r)\cdot(i' + pt_2)p^{j'-1}\equiv{0}\bmod{p^e}.$$

Recall that any solution suffices, since $\ell=n-1$. Write $i' = i + \eta$ and that $j' = j + \varepsilon$, so that $\varepsilon=0$ if and only if $i=i'$, and $\eta=0$ if and only if $j=j'$. In particular, at least one of $\eta$ and $\varepsilon$ must be nonzero, by hypothesis. Assume without loss of generality that $\eta\geq 0$ (that is, $j'\geq j$). Then the critical equivalence becomes
\begin{align*}
    r\cdot(i + pt_1)p^{j-1} + (p^e - 1 - r)\cdot(i + \eta + pt_2)p^{j+\varepsilon-1} &\equiv{0}\bmod{p^e}  \\
    r(i+pt_1-p^{\varepsilon}(i+\eta+pt_2)) - (i+\eta+pt_2) &\equiv 0\bmod{p^{e-(j-1)}}.
\end{align*}

In the general case when $\varepsilon\neq 0$, we have that $i+pt_1-p^{\varepsilon}(i+\eta+pt_2)$ is invertible modulo $p^{e-j+1}$, since $1\leq i\leq p-1$ and so $\gcd(i,p)=1$. Thus the critical equivalence has a solution modulo $p^{e-j+1}$, namely
    $$r\equiv\frac{i + pt_1}{p^{\varepsilon}(i + \eta + pt_2) - i - pt_2}\bmod{p^{e-j+1}}.$$
In the exceptional case when $\varepsilon=0$, we can solve the critical equivalence in a similar way, this time using $\gcd(\eta,p)=1$:
\begin{align*}
    r(i+pt_1-(i+\eta+pt_2)) - (i+\eta+pt_2) &\equiv 0\bmod{p^{e-(j-1)}} \\
    r(\eta+p(t_1-t_2)) - (i+\eta+pt_2) &\equiv 0 \bmod{p^{e-(j-1)}}
\end{align*}
    $$r \equiv\frac{i+\eta+pt_2}{\eta+p(t_1-t_2)} \bmod{p^{e-j+1}}.$$
In either case, the critical equivalence has a solution modulo $p^{e-j+1}$, and hence a valid solution modulo $n=p^e$. This shows that $\Delta_{p^e,p^e-1}$ is a disjoint union of the simplices on vertex sets $V_{i,j}$, and thus concludes the proof of part (a).
    
For part (b), let $\inlatA$ be the corresponding intersection poset of subspaces. Since $\Delta_{p^e,p^e-1}$ is a collection of $e(p-1)$ many $(p^j-1)$-simplices for each $0\leq j\leq e-1$, then by \textbf{Proposition \ref{max-chain-union}}, $\inlatA$ is graded of rank $2$.

For part (c), observe that $\Delta_{p^e,p^e-1}$ has $p^e-1$ vertices, and for any facet $|V_{i,j}|=p^{e-j}$. Since $i$ ranges from $1$ to $p$, and $j$ ranges from $1$ to $e$, \textbf{Proposition \ref{char-poly-union}} computes that $\chi_\mathcal{A}(x)= x^{p^e-1}-(p-1)\sum^{e-1}_{j=0}x^{p^{j}}$, as desired.
\end{proof}
    
\subsection*{Proof of Theorem \ref{total-armslegs}}

There are several calculations that come up repeatedly in part (a) of the theorem. We remark that the vertices of $\Delta_{2p,2p-s}$ are all the integers between $1$ and $2p-1$, except $p$ is excluded for even $s$.

\begin{lemma}
\label{arms-legs-simplices-lemma}
Consider $\Gamma_s = \Delta_{2p,2p-s}$ for prime $p\geq{3}$ and odd $s\leq 2p$. Then the set of odd vertices is a face of $\Gamma_s$.
\end{lemma}
\begin{proof}
Take any collection of $2p-s$ many (not necessarily distinct) odd vertices $\{2t_1 + 1,2t_2 + 1,\cdots,2t_{2p-s} + 1\}$.
Then
$$\sum_{j=1}^{2p-s}(2t_j + 1) = 2\sum_{j=1}^{2p-s}(t_j) + 2p-s = 2\left(\sum_{j=1}^{2p-s}t_j - \frac{s-1}{2}\right) - 1,$$
which is an odd element of $\mathbb{Z}/2p\mathbb{Z}$, and thus nonzero. Hence, any subset of odd vertices must be a face of $\Gamma_s$; in particular, the set of all odd vertices is a face of $\Gamma_s$.
\end{proof}

\begin{lemma}
\label{no-even-vertices-arms-legs}
Consider $\Gamma_s = \Delta_{2p,2p-s}$ for $p$ prime and $s\leq p$. Then no two even vertices are adjacent in $\Gamma_s$; moreover, if $s$ is even, then no two odd vertices are adjacent either.
\end{lemma}

\begin{proof}
Two even vertices $v$ and $w$ are adjacent in $\Gamma_s$ if and only if the critical equivalence  
 $$r(v) + (2p-s-r)(w)\equiv{0}\bmod{2p}$$
has no solution $0\leq{r}\leq{2p-s}$. Given two distinct even vertices $v=2t_1$ and $w=2t_2$, we can manipulate the critical equivalence to $2r(t_1-t_2)\equiv 2st_2\bmod{2p}$. Since $v\not\equiv w\bmod{2p}$, we have $t_1-t_2\not\equiv{0}\bmod{p}$. Thus $t_1 - t_2$ is invertible modulo $p$, and so $r$ is a solution if and only if
    $$r\equiv{\frac{st_2}{t_1 - t_2}}\bmod{p}.$$
There is necessarily such an $r$ in the range $0\leq r<p\leq 2p-s$, and so no two even vertices are adjacent.

Now consider the case where $s$ is even, and let $\frac{s}{2} = q$. For any odd vertices $v = 2t_1 + 1$ and $w = 2t_2 + 1$, we can manipulate the critical equivalence to $2r(t_1-t_2)\equiv 2q(2t_2+1)\bmod{2p}$.
As before, $t_1-t_2$ is invertible modulo $p$, and hence $r$ is a solution if and only if
    $$r\equiv\frac{q(2t_2 + 1)}{t_1-t_2}\bmod{p}.$$
As before, there is such an $r$ in the range $0\leq{r}<p\leq{2p-s}$, and thus no two odd vertices are adjacent in this case. 
\end{proof}

\begin{lemma}
\label{arms-lemma}
Fix a prime $p$, a positive integer $s<p$, an odd element $i\in\mathbb{Z}/2p\mathbb{Z}$ and a nonzero even element $j\in\mathbb{Z}/2p\mathbb{Z}$. Then a solution exists to the critical equivalence
    $$ri + (2p-s-r)j\equiv{0}\bmod{2p}$$
if and only if $i\not\equiv{j}\bmod{p}$. In this case, all solutions satisfy 
    $$r\equiv\frac{sj}{i-j}\bmod{2p}.$$
\end{lemma}
\begin{proof}
Begin by observing that $r$ is a solution to the critical equivalence if and only if $r(i-j) \equiv{sj}\bmod{2p}$.

First, suppose that $i\not\equiv{j}\bmod{p}$. Note that since $i-j\not\equiv{0}\bmod{p}$ and is also odd, we have that $\gcd(i-j,2p) = 1$, and thus $i-j$ has a multiplicative inverse modulo $2p$. Then a solution to the critical equivalence exists, and any solution satisfies 
$r\equiv\frac{sj}{i-j}\bmod{2p}$.

Conversely, suppose that $i\equiv{j}\bmod{p}$. Then $i-j$ is odd, but $sj$ is even, and thus $r$ must be even; say $\frac{r}{2} = q$. Thus $r$ is a solution if and only if $q(i-j)\equiv{sj}\bmod{p}$. Since $i-j\equiv{0}\bmod{p}$, and since $j\not\equiv{0}\bmod{p}$ and $s\not\equiv{0}\bmod{p}$ (as $s<p$), it follows that there exists no solution to the critical equivalence when $i\equiv{j}\bmod{p}$.
\end{proof}

We are now ready to prove the theorem.

\armslegs*


\begin{proof} For part (a), we make extensive use of the lemmata above.
\\

\caseA The set of odd vertices is a face of $\Gamma_1$ by \textbf{Lemma \ref{arms-legs-simplices-lemma}}. By \textbf{Lemma \ref{no-even-vertices-arms-legs}}, we have that no even vertices are adjacent in $\Gamma_1$.

Since $\ell=n-1$, an odd vertex $i$ is adjacent to an even vertex $j$ if and only if the critical equivalence has no solution. By \textbf{Lemma \ref{arms-lemma}}, a solution to this equivalence exists if and only if $i\equiv{j}\bmod{p}$.
\\

\caseB By \textbf{Lemma \ref{no-even-vertices-arms-legs}}, no two even vertices are adjacent in $\Gamma_2$ and no two odd vertices are adjacent in $\Gamma_2$ either. By \textbf{Lemma \ref{arms-lemma}}, an odd vertex $i$ is adjacent to an even vertex $j$ if and only if $i\equiv{j}\bmod{p}$, except in the case when the only solutions to the critical equivalence are invalid. We show that this case cannot occur.

Simplifying the critical equivalence for $\Gamma_2$ yields that $i$ and $j$ are adjacent if and only if there is no solution to $r(i - j)\equiv 3j\bmod{2p}$ with $1\leq r\leq 2p-2$. Observe that if $r\equiv -1 \bmod{2p}$, then $i\equiv{-j}\bmod{2p}$. In particular, if $i$ is odd and $j$ is even, then $r\not\equiv{-1}\bmod{2p}$ is not a solution to the critical equivalence, and so by \textbf{Lemma \ref{arms-lemma}} there is a solution if and only if there is a valid solution.
\\

\caseC By \textbf{Lemma \ref{arms-legs-simplices-lemma}}, the set of odd vertices is a face of $\Gamma_3$, and by \textbf{Lemma \ref{no-even-vertices-arms-legs}}, no even vertices in $\Gamma_3$ are adjacent. By \textbf{Lemma \ref{arms-lemma}}, any odd vertex $i$ and even vertex $j$ are adjacent when $i\equiv{j}\bmod{p}$, since no solution to the critical equivalence exists. 

Now, suppose that $i\not\equiv{j}\bmod{p}.$ Then by \textbf{Lemma \ref{arms-lemma}}, it follows that a solution to the critical equivalence exists. We wish to determine for which $i$ and $j$ this solution is valid; by \textbf{Lemma \ref{arms-lemma}} we need only check that any fixed solution is equivalent to neither $-1$ nor $-2$ modulo $2p$. Simplifying the critical equivalence for $\Gamma_2$ yields $r(i - j)\equiv 3j\bmod{2p}$. So if we suppose a solution satisfies $r\equiv{-1}\bmod{2p}$, we may conclude that $i\equiv{-2j}\bmod{2p}$. This is a contradiction as $i$ is assumed to be odd. On the other hand, if we suppose a solution satisfies $r\equiv{-2}\bmod{p}$, then we conclude that $-2i\equiv{j}\bmod{2p}$. 

Thus, it follows that the vertices $i$ and $j$ are connected when $j\equiv{-2i}\bmod{2p}$, or $j\equiv{i}\bmod{p}.$ 
\\

We quickly complete the proof for $s=2$ before turning to the more interesting cases: 

\caseB From part (a) we see that $\Delta_{2p,2p-2}$ is a disjoint union of $p-1$ edges, which are $1$-simplices. So by \textbf{Proposition \ref{max-chain-union}}, $\mathcal{L}(\mathcal{A}_2)$ is graded of rank $2$, which completes the proof of part (b). Moreover, by \textbf{Corollary \ref{char-poly-union-eqdim}} we have $\chi_\mathcal{A}(x)=x^{2p-2}-px^{2}+p-1$, which completes the proof of part (c).
\\

For $s=1$ and $s=3$, we know from part (a) that the set $F_0$ of $(p-1)$ odd vertices is a facet of $\Delta_{2p,2p-s}$ for $s\in\{1,3\}$, and all other facets (including for $s=2$) are of the form $F=\{f_1,f_2\}$ where $f_2$ is even and $f_1$ is an odd number distinct from $p$. In particular, for any $F\neq F_0$, we have $F\cap F_0 = \{f_1\}$. We denote the corresponding subspaces of $\mathcal A_s$ by $S_{F_0}$, $S_F$ and $S_{f_1}$, which have dimensions $p-1$, $2,$ and $1$ respectively. To resolve part (b), we observe:
\\

\caseA From part (a) we see there are $p$ facets, and all the two-element facets are disjoint. It is easily checked that every intersection of two subspace of $\mathcal A_1$ is either $\hat 1$ or $S_{f_1}=S_{F_0}\cap S_F$ for some unique $F$. Moreover, the intersection of any three subspaces of $\mathcal A_1$ is $\hat 1$. Thus any maximal chain has the form $\{\hat 0=\mathbb{K}^{2p-1}\leq S_F \leq S_{f_1} \leq \{0\}=\hat 1\}$.
\\

\caseC From part (a) we see there are there are $2p-1$ facets. In addition to the intersections above, we also have that for each even $v\in V,$ there exists a unique pair of facets $F,G$, neither of which are $F_0$, such that $F\cap G=\{v\}$. We denote these by $S_{v}$; they have dimension $1$. It is easily checked that the intersection of any two subspaces in $\mathcal A_3$ is either $\hat 1$ or $S_v$ for some $v\in V$, the intersection of any three subspaces is either $\hat 1$ or $S_{f_1}$ for some odd $f_1\in V$ (recall, in particular, that $p\notin V$), and the intersection of any four subspace is $\hat 1$. Thus, any maximal chain has the form $\{\hat 0=\mathbb{K}^{2p-1}\leq S_F \leq S_{v} \leq \{0\}=\hat 1\}$.
\\

Hence, $\mathcal L(\mathcal A_1)$ and $\mathcal L(\mathcal A_1)$ are each graded of rank $3$, completing the proof of part (b). Proceeding to part (c), we note that $\mu(\hat{0},\hat{0})=1$, and $\mu(\hat{0},S_F)=\mu(\hat{0},S_{F_0})=-1$ for each facet $F$ of dimension $1$, for both $\mathcal L(\mathcal A_1)$ and $\mathcal L(\mathcal A_3)$. Moreover:
\\

\caseA For $s=1$, each $S_{f_1}$ is contained in $S_{F_0}$ and exactly one $S_F$. So, $\mu(\hat{0},S_{f_1})=-(1-2)=1$. Finally, we have $\mu(\hat{0},\hat{1})=-(1+-p+(p-1))=0$. Thus, by definition of the characteristic polynomial, $\chi_{\mathcal{A}_1}(x)=x^{2p-1}-x^p-(p-1)x^2+(p-1)x$. 
\\

\caseC For $s=3$ each $S_v$ is contained in exactly two $S_F$, and if $v$ is odd it is also contained in $S_{F_0}$. So, $\mu(\hat{0},S_{v})$ is $1$ if $v$ is even, and $2$ if $v$ is odd. Finally, we have $\mu(\hat{0},\hat{1})=-(1+-(2p-1)+(p-1)+2(p-1))=-(p-1)$. Therefore, the characteristic polynomial is $\chi_{\mathcal{A}_3}(x)=x^{2p-1}-x^p-2(p-1)x^2+3(p-1)x-(p-1)$

\end{proof}

\section*{Acknowledgments}
 
The authors are deeply grateful for the mentorship of Kaisa Taipale. We would also like to thank Ryan Matzke and Vic Reiner for many enlightening conversations, and the University of Minnesota for facilitating our collaboration. This work was partially supported by NSF RTG grant DMS-1745638.

\bibliographystyle{plain}
\bibliography{AHS-arxiv-draft} 

\begin{thebibliography}{1}

\bibitem{Athanasiadis}
Christos~A. Athanasiadis.
\newblock Characteristic polynomials of subspace arrangements and finite
  fields.
\newblock {\em Advances in Mathematics}, 122(2):193--233, 1996.

\bibitem{menu-research-problems}
B\'ela Bajnok.
\newblock {\em Additive Combinatorics}.
\newblock Chapman and Hall/CRC, 2018.

\bibitem{BajnokMatzke}
B\'ela Bajnok and Ryan Matzke.
\newblock The maximum size of $(k,l)$-sum-free sets in cyclic groups.
\newblock {\em Bulletin of the Australian Mathematical Society}, pages
  184--194, 2019.

\bibitem{BierChin}
T.~Bier and A.~Y.~M. Chin.
\newblock On $(k, l)$-sets in cyclic groups of odd prime order.
\newblock {\em Bulletin of the Australian Mathematical Society},
  63(1):115--121, 2001.

\bibitem{CalkinThomson}
Neil~J Calkin and Jan~McDonald Thomson.
\newblock Counting generalized sum-free sets.
\newblock {\em Journal of Number Theory}, 68(2):151--159, 1998.

\bibitem{CameronErdos}
P.~Cameron and P.~Erd\H{o}s.
\newblock On the number of sets of integers with various properties.
\newblock {\em Number Theory: Proceedings of the 1988 Canadian Number Theory
  Conference at Banff}, pages 61--80, 01 1990.

\bibitem{Green}
Ben Green.
\newblock The {C}ameron-{E}rd\h{o}s conjecture.
\newblock {\em Bulletin of the London Mathematical Society}, 36(6):769--778,
  2004.

\bibitem{HamidounePlagne}
Yahya Hamidoune and Alain Plagne.
\newblock A new critical pair theorem applied to sum-free sets in abelian
  groups.
\newblock {\em Commentarii Mathematici Helvetici}, 79(1):183--207, Jan 2004.

\bibitem{Sapozhenko}
Alexander~A. Sapozhenko.
\newblock The {C}ameron-{E}rd\h{o}s conjecture.
\newblock {\em Discrete Mathematics}, 308(19):4361--4369, 2008.
\newblock Simonovits '06.

\end{thebibliography}

\end{document}